\documentclass[12pt]{amsart}

\numberwithin{equation}{section}
\numberwithin{figure}{section}

\usepackage{tikz-cd}
\usepackage{amssymb}
\usepackage{mathtools}
\usepackage[headings]{fullpage}
\usepackage{amsmath}

\usepackage{color}
\usepackage[normalem]{ulem}
\usepackage[colorlinks=true,linkcolor=blue, citecolor=blue, urlcolor=blue, 
pagebackref=true]{hyperref}
\usepackage{enumitem}
\usepackage{tikz}
\usepackage{xparse}

\newcount\tableauRow
\newcount\tableauCol

\usetikzlibrary {positioning}
\definecolor {processblue}{cmyk}{0.96,0,0,0}

\makeatletter
\def\theenumi{\@alph\c@enumi}
\makeatother

\usepackage{units}
\usepackage{ytableau}

\theoremstyle{plain}
\newtheorem{theorem}[equation]{Theorem}
\newtheorem{lemma}[equation]{Lemma}
\newtheorem{corollary}[equation]{Corollary}
\newtheorem{proposition}[equation]{Proposition}
\theoremstyle{definition}

\newtheorem{remark}[equation]{Remark}

\newtheorem{example}[equation]{Example}

\newtheorem{definition}[equation]{Definition}

\newtheorem{notation}[equation]{Notation}

\newtheorem{discussion}[equation]{Discussion}

\newtheorem{observation}[equation]{Observation}

\newtheorem{construction}[equation]{Construction}

\newcommand{\GL}{\mathrm{GL}}

\newcommand{\sign}[1]{\mathrm{sgn(#1)}}

\newcommand{\rcf}[2]{\mathcal{R}_{\scriptscriptstyle #1,#2}}
\newcommand{\stdcontent}[2]{S(#1,#2)}
\newcommand{\tabcontent}[2]{T(#1,#2)}
\newcommand{\fillcontent}[2]{F(#1,#2)}
\newcommand{\fillcontentc}[2]{F^c(#1,#2)}
\newcommand{\rbasis}[1]{D_{#1}}
\newcommand{\reval}[1]{\mathfrak{R}_{-,#1}}
\newcommand{\kostka}[2]{K_{\scriptscriptstyle #1,#2}}

\newcommand{\perm}[2]{#1(#2)}
\newcommand{\ycs}[2]{C_{\scriptscriptstyle #2}(#1)}
\newcommand{\sorting}[1]{\operatorname{sort}(#1)}
\newcommand{\rowsorting}[1]{\operatorname{rowsort}(#1)}

\title{A non-iterative formula for straightening fillings of Young diagrams}
\author{Reuven Hodges}
\address{Department of Mathematics, University of Kansas, Lawrence, KS 66045, USA}
\email{rmhodges@ku.edu}

\begin{document}

\begin{abstract}
Young diagrams are fundamental combinatorial objects in representation theory and algebraic geometry. Many constructions that rely on these objects depend on variations of a straightening process that expresses a filling of a Young diagram as a sum of semistandard tableaux subject to certain relations. This paper solves the long standing open problem of giving a non-iterative formula for straightening a filling. We apply our formula to give a complete generalization of a theorem of Gonciulea and Lakshmibai. 
\end{abstract}
\maketitle

\section{Introduction} 
\label{sec:intro}
The textbook constructions of the irreducible representations of the symmetric group and the polynomial irreducible representations of the general linear group rely on a process called straightening. This straightening process expresses a filling of a Young diagram as a sum of semistandard tableaux subject to certain relations. The aim of this paper is to provide a non-iterative formula for straightening a filling, as well as a combinatorial description of the coefficients that appear. 

As far as the author is aware, all previously discovered straightening algorithms (\cite{MR1035487}, \cite[\S 2.6]{MR1824028}, \cite[\S 3.1]{MR2667486}, \cite[\S 15.5]{MR1153249}) are inductive. Consequently, they give little combinatorial control over the coefficients that arise when straightening. For example, in \cite[Theorem 5.2]{MR1417711}, Gonciulea-Lakshmibai perform a technical analysis of an inductive straightening algorithm to show that a specific coefficient is nonzero when straightening a tableau with shape equal to a two column rectangular Young diagram. This is the key result that the authors employ to show that Schubert varieties in a miniscule flag variety degenerate to normal toric varieties.  An immediate corollary of our non-iterative formula is a generalization of their theorem to fillings of arbitrary shape.

\subsection{Non-iterative straightening}
Let $R$ be a commutative ring. Define $\fillcontent{\lambda}{z}$ to be the set of fillings of the Young diagram of shape $\lambda$ with content $z$, where the \emph{content} of a filling is a composition that records the number of occurrences of each value in the filling. Then $R^{\fillcontent{\lambda}{z}}$ is defined to be the free $R$-module generated by $\fillcontent{\lambda}{z}$. Let $A(\lambda,z)$ be the submodule of $R^{\fillcontent{\lambda}{z}}$ generated by the Grassmannian and Pl\"{u}cker relations (see Section \ref{subsec:factSpace}). The equivalence classes of semistandard tableaux form a basis of the quotient module $R^{\fillcontent{\lambda}{z}} / A(\lambda,z)$. Expressing a filling as a $\mathbb{Z}$-linear combination of semistandard tableaux in this quotient is called \emph{straightening} the filling.

Given a partition $\lambda$ we consider the group $\ycs{\lambda}{}$ of permutations that permute the entries of a given filling of shape $\lambda$ within each column. Let $F$ and $S$ be two fillings of shape $\lambda$. For $\underline{\pi} \in \ycs{\lambda}{}$ we define $F_{\underline{\pi}}$ to be the filling that results from permuting the entries of $F$ according to $\underline{\pi}$. The \emph{rearrangement coefficient} of $F$ with respect to $S$, denoted $\rcf{F}{S}$, is the sum of the signs of all $\underline{\pi} \in \ycs{\lambda}{}$ such that $F_{\underline{\pi}}$ has the same content as $S$ in each row. The fundamental insight of this paper is that the straightening of a filling depends only on these rearrangement coefficients.   

Let $\stdcontent{\lambda}{z}$ be the subset of $\fillcontent{\lambda}{z}$ containing the semistandard tableaux of shape $\lambda$ and content $z$. Order and label all semistandard tableaux in $\stdcontent{\lambda}{z}$  as $S_1 \succ S_2 \succ \cdots \succ S_{\kostka{\lambda}{z}}$ where $\kostka{\lambda}{z}$ is the Kostka number and $\succ$ is the reading word order (see Definition \ref{definition:rwCwOrder}). We recursively construct a new basis of $R^{\fillcontent{\lambda}{z}} / A(\lambda, z)$, called the \emph{D-basis}, by setting
\begin{center}
$\rbasis{S_{i}} := S_{i} - \displaystyle \sum_{\substack{S_j \in \stdcontent{\lambda}{z} \\ \textrm{such that } j < i}} \rcf{S_{i}}{S_j} \cdot \rbasis{S_j}$
\end{center}
for each $S_i \in \stdcontent{\lambda}{z}$. Our main result is that straightening a filling in terms of this new basis is considerably simpler.  

\begin{theorem}
\label{theorem:mainTheoremStr1In}
Let $F \in \fillcontent{\lambda}{z}$, with $F=\sum_ {S_i \in \stdcontent{\lambda}{z}}  a_i S_i$ in $R^{\fillcontent{\lambda}{z}} / A(\lambda, z)$. Then
\begin{center}
$\displaystyle \sum a_i S_i = \displaystyle \sum_{S_j \in \stdcontent{\lambda}{z}} \rcf{F}{S_j} \cdot \rbasis{S_j}$
\end{center}
\end{theorem}

If $F=\sum_ {S_i \in \stdcontent{\lambda}{z}}  a_i S_i$ is nonzero, we say that the \emph{leading term} of $F$ is the $S_k \in \stdcontent{\lambda}{z}$ such that $a_k \neq 0$ and $a_{\ell} = 0$ for all $\ell > k$.  As alluded to above, in \cite[Theorem 5.2]{MR1417711}, Gonciulea and Lakshmibai prove that the coefficient of the leading term is equal to 1 for a tableau whose shape is a two column rectangle. We give a generalization of this result to any filling of any shape. Let $\sorting{F}$ be the filling that arises from $F$ by first reordering the values within the columns of $F$ so that they increase weakly downwards, and then reordering the values within each row so that they increase weakly left to right. If $F$ has no duplicated values in each of its columns then we call it a \emph{cardinal filling}, and Lemma~\ref{lemma:sortingSST} implies that $\sorting{F}$ is a semistandard tableau.

\begin{corollary}
\label{corollary:mainCorollaryStr2}
Let $F \in \fillcontent{\lambda}{z}$ be a cardinal filling, with $\sorting{F}=S_k \in \stdcontent{\lambda}{z}$. If $F\!=\!\sum_ {S_i \in \stdcontent{\lambda}{z}}  a_i S_i$ in $R^{\fillcontent{\lambda}{z}} / A(\lambda, z)$, then $S_k$ is the leading term.\! Further,\! $a_k\!=\!\sign{\underline{\sigma}}$ where $\underline{\sigma}$ is the unique permutation in $C(\lambda)$ such that $F_{\underline{\sigma}}$ is a tableau. If $F$ is a tableau, $a_k\!=\!1$.
\end{corollary}

\subsection{The homogeneous coordinate rings of flag varieties}
Straightening plays an important role in understanding the homogeneous coordinate rings of flag varieties and their Schubert subvarieties. In 1943, Hodge\cite{MR0007739} gave a basis for the homogeneous coordinate rings of the complex Grassmannian and its Schubert subvarieties in terms of Young diagrams and tableaux whose proof relied on straightening. In particular, straightening describes multiplication in this ring. Then, in the 1970s, Seshadri, Lakshmibai, and Musili began to work on extending the results of Hodge to all varieties of the form $G/P$, where $G$ is a semisimple algebraic group and $P$ is a parabolic subgroup. They were successful for $G$ of classical type, and formulated several conjectures for a general semisimple $G$ that were proved by Littelmann in \cite{MR1670770}. Straightening played a crucial role in the development of this theory and a more detailed account may be found in \cite{MR2388163}.

\subsection{Complexity}
Straightening a filling using Theorem~\ref{theorem:mainTheoremStr1In} is a two step process. The first step is computing the $D$-basis associated to the shape and content, and the second is computing at most $\kostka{\lambda}{z}$ rearrangement coefficients. Even better, the $D$-basis computation only depends on the shape and content, and hence straightening multiple fillings of the same shape and content is computationally efficient. 

While preparing this paper, the author implemented several traditional straightening algorithms as well as this non-iterative method in the programming language C and significant heuristic evidence suggests that the non-iterative method is several orders of magnitude faster. In light of this, we plan to study the computational complexity of straightening in a subsequent work.

One of the most frequent computational applications of straightening is to decide if two vectors in the quotient module are linearly independent (see \cite[Remark 12]{MR3574536}). Given two fillings $E, F \in \fillcontent{\lambda}{z}$, their coordinate vectors with respect to the $D$-basis may be found by computing rearrangement coefficients. The linear independence of these vectors may be discerned via their coordinate vectors. Importantly, all of this can all be achieved without ever explicitly computing the $D$-basis.

\subsection{Organization}
The content of the paper is divided as follows. Section 2 gives the necessary background on partitions, Young diagrams, fillings, and semistandard tableaux. These definitions are used to give a construction of the quotient module, and this in turn is used to give an explicit construction of the irreducible representations of the symmetric and general linear groups. In Section 3 the rearrangement coefficients and the D-basis are introduced and a number of simple lemmas related to these are proved. The rest of the section is devoted to proving that for each $S \in \stdcontent{\lambda}{z}$ there exists a $R$-module homomorphism $\reval{S}:R^{\fillcontent{\lambda}{z}} / A(\lambda, z) \rightarrow R$ that maps $F \in F(\lambda, z)$ to $\rcf{F}{S}$. The primary results of the paper are in Section 4. Proposition \ref{proposition:mainStraighten} is proved using the results from the previous sections and is then used to show the main result, Theorem \ref{theorem:mainTheoremStr1In}. The paper concludes by applying Theorem \ref{theorem:mainTheoremStr1In} to give a combinatorial description of the straightening coefficients and visualization of these coefficients in terms of paths in a directed graph.

\section*{Acknowledgements}
The author would like to thank Venkatramani Lakshmibai and Alexander Yong for helpful discussions, comments, and suggestions. The author would also like to thank the anonymous referee for suggestions that significantly simplified the main argument and improved the overall clarity.

\section{Preliminaries}
\label{sec:prelim}
In this section we recall a number of standard definitions and results on partitions, Young diagrams, fillings, and straightening.  We follow the formulation in \cite[Chapter 4.1]{phdthesisCIkenmeyer}, for more details and proofs see \cite{MR1464693} and \cite{MR1153249}.

\subsection{Partitions and Fillings}
\label{subsec:partsFillings}
Fix $n\geq 2$ and let $[n]$ denote the set $\{ 1,\ldots,n \}$. The set of nonnegative integers will be denoted $\mathbb{N}$ while the set of positive integers will be denoted $\mathbb{N}_{>0}$.

A \emph{partition} is a finite sequence of positive integers $\lambda \coloneqq (\lambda_1,\ldots,\lambda_k)$ such that $\lambda_1 \geq \cdots \geq \lambda_k > 0$. We define the \emph{size} of the partition $\lambda$ as $|\lambda|=\sum \lambda_i$ and say that the \emph{length}, denoted $l(\lambda)$, of the partition is the number $k$ of values in the sequence. A partition $\lambda$ is often visualized using its \emph{Young diagram}, also denoted $\lambda$, an upper left justified collection of boxes, with $\lambda_i$ boxes in the $i$th row. We identify partitions with their Young diagrams and say that the Young diagram associated to the partition $\lambda$ has \emph{shape} $\lambda$. The column lengths of $\lambda$ are denoted by $\zeta_1,\ldots,\zeta_{\lambda_1}$, and $\zeta \coloneqq (\zeta_1,\ldots,\zeta_{\lambda_1})$ is called the \emph{conjugate partition}. A location in $\lambda$ with column index $1 \leq c \leq \lambda_1$ and row index $1 \leq r \leq \zeta_c$ may be specified by $(r,c)$. We use the notation $a \times b$ to denote the sequence $(b,\ldots,b)$ of $a$ copies of $b$. 

\begin{example}
\label{example:youngDiagram}
The partition $(4,2,2)$ is identified with the following Young diagram.
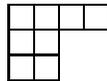
\begin{center}
\ytableausetup{mathmode,boxsize=1em,centertableaux}
\begin{ytableau}
\; & \; & \; & \; \\
\; & \; \\
\; & \; \\ 
\end{ytableau}
\end{center}
\end{example}  

Given a partition $\lambda$, a \emph{filling} $F$ of shape $\lambda$ with entries in $[n]$ is an assignment of a value in $[n]$ to each location in the Young diagram $\lambda$. We write $F(r,c)$ for the value of a filling $F$ at the location $(r,c)$. A filling such that no column contains two equal values is called a \emph{cardinal filling}. A \emph{tableau} $T$ of shape $\lambda$ with entries in $[n]$ is a cardinal filling such that the values increase strictly down each column. A \emph{semistandard tableau} $S$ of shape $\lambda$ with entries in $[n]$ is a tableau such that the values increase weakly along each row. The \emph{content} of a filling is the $n$-tuple of non-negative numbers $z=(z_1,\ldots,z_n)$ where $z_i$ is equal to the number of locations with value equal to $i$ in the filling. A filling is called a \emph{numbering} of shape $\lambda$ if the content of the filling equals $|\lambda| \times 1$. A semistandard tableau that is also a numbering is called a \emph{standard tableau}.

\begin{example}
\label{example:youngDiagram2}
From left to right: a filling, cardinal filling, tableau, semistandard tableau, and standard tableau of shape $(4,2,2)$.
\begin{center}
\ytableausetup{mathmode,boxsize=1em,centertableaux}
\begin{ytableau}
1 & 1 & 2 & 1 \\
2 & 2 \\
1 & 2 \\
\end{ytableau}
\qquad 
\begin{ytableau}
3 & 1 & 2 & 2 \\
2 & 3 \\
1 & 2 \\
\end{ytableau}
\qquad 
\begin{ytableau}
2 & 1 & 4 & 1 \\
3 & 2 \\
4 & 3 \\
\end{ytableau}
\qquad 
\begin{ytableau}
1 & 2 & 3 & 3 \\
3 & 3 \\
4 & 5 \\
\end{ytableau}
\qquad 
\begin{ytableau}
1 & 2 & 3 & 6 \\
4 & 7 \\
5 & 8 \\
\end{ytableau}
\end{center}
\end{example}

If $E$ and $F$ are two fillings of shape $\lambda$, we say that they have the same \emph{row content} if the multiset of values in row $i$ of $E$ is equal to the multiset of values in row $i$ of $F$ for all $i \in [\zeta_1]$.
\begin{notation} Let $\fillcontent{\lambda}{z}$ be the finite set of fillings of shape $\lambda$ with entries in $[n]$ and content $z$. Denote by $\fillcontentc{\lambda}{z}$ the subset of $\fillcontent{\lambda}{z}$ containing the cardinal fillings. We will write $\tabcontent{\lambda}{z}$ for the subset of $\fillcontentc{\lambda}{z}$ containing those fillings that are tableaux, and $\stdcontent{\lambda}{z}$ for the subset of semistandard tableaux. With this notation in place we have the following inclusions
\begin{center}
$\stdcontent{\lambda}{z} \subset \tabcontent{\lambda}{z} \subset \fillcontentc{\lambda}{z} \subset \fillcontent{\lambda}{z}$
\end{center}
for a fixed partition $\lambda$ and content $z$. \end{notation}

\begin{definition}
\label{definition:rwCwOrder}
We define a total order called the \emph{reading word order} on the set $\fillcontent{\lambda}{z}$. Let $E, F \in \fillcontent{\lambda}{z}$. We define the \emph{reading word} $\textrm{rw}(F)$ to be the word formed by listing the values in the filling $F$ read from left to right in each row, starting with the top row and moving downward. We write $E \preceq F$ if $\textrm{rw}(E)$ precedes $\textrm{rw}(F)$ lexicographically.
\end{definition}

For $F \in \fillcontent{\lambda}{z}$, we define the \emph{row-sorting} of $F$, denoted $\rowsorting{F}$, to be the filling that arises from $F$ by reordering the values within the rows so that they increase weakly.

For a filling $F \in \fillcontent{\lambda}{z}$, define the \emph{sorting} of $F$ to be the filling that arises from first reordering the values within the columns of $F$ so that they increase weakly downwards, and then reordering the values within the rows so that they increase weakly along each row. We denote the sorting of $F$ by $\sorting{F}$ and recall the following lemma.
\begin{lemma}
\label{lemma:sortingSST}
Let $F \in \fillcontentc{\lambda}{z}$, then $\sorting{F}$ is a semistandard tableau in $\stdcontent{\lambda}{z}$.
\end{lemma}
This lemma is a variant of the ``Non-Messing-Up'' Theorem\cite{MR2311934}. For a modern proof, with formulation that more closely matches ours, see \cite[Proposition 4.1]{MR2596377}. Although the result in \cite{MR2596377} is for numberings, its proof can be extended to cardinal fillings by changing some $<$ signs into $\leq$ signs in the proof. An earlier proof can be found in \cite[Lemma 1]{MR561764}.

\begin{example}
Let 
\begin{center}
$\ytableausetup{mathmode,boxsize=1em,centertableaux}
E=\begin{ytableau}
2 & 1 & 3 & 4 \\
3 & 2 \\
4 & 3 \\
\end{ytableau}
\qquad \textrm{and} \qquad
F=\begin{ytableau}
2 & 2 & 4 & 3 \\
3 & 1 \\
3 & 4 \\
\end{ytableau}$
\end{center}
be two fillings. Then $\textrm{rw}(E)=21343243$ and $\textrm{rw}(F)=22433134$. Thus $E \prec F$ since $21343243$ precedes $22433134$ lexicographically. The sorting of these two fillings are equal with
\begin{center}
\ytableausetup{mathmode,boxsize=1em,centertableaux}
$\sorting{E}=\sorting{F}=\begin{ytableau}
1 & 2 & 3 & 4 \\
2 & 3 \\
3 & 4 \\
\end{ytableau}$
\end{center}
\end{example}
 
\subsection{The quotient module}
\label{subsec:factSpace}
Fix a partition $\lambda$ and content $z$ with $|\lambda|=|z|$. Let $R^{\fillcontent{\lambda}{z}}$ be the free $R$-module generated by $\fillcontent{\lambda}{z}$. Let $A(\lambda,z)$ be the $R$-submodule of $R^{\fillcontent{\lambda}{z}}$ generated by
\begin{enumerate}[label=(\roman*)]
\item elements $E$, where $E$ is a non-cardinal filling in $\fillcontent{\lambda}{z}$;
\item sums $E+F$, where $E$, $F \in \fillcontent{\lambda}{z}$ are two fillings that differ in a single column by a single transposition of values;
\item differences $E - \sum_{F} F$, where $E \in \fillcontent{\lambda}{z}$ and the sum is over all fillings $F$ that result from $E$ by exchanging the $m$ top elements from the $(j+1)$th column with any $m$ elements in the $j$th column (maintaining their vertical order), for some fixed positive integers $j$ and $m$. For each such difference, the integers $j$ and $m$ are fixed, but the choice of $m$ elements in the $j$th column varies in the sum.
\end{enumerate}
The sums in (ii) are called the \emph{Grassmannian relations} and the differences in (iii) are called the \emph{Pl\"{u}cker relations}. 

\begin{example}
Let $\lambda = (2,2,1)$. Then
\begin{center}
\ytableausetup{mathmode,boxsize=1em,centertableaux}
\begin{ytableau}
    2 & 1 \\
    3 & 4 \\
    4 \\
    \end{ytableau} +
    \begin{ytableau}
    3 & 1 \\
    2 & 4 \\
    4 \\ 
    \end{ytableau}
\end{center}
is a Grassmannian relation. Fixing $m=2$ and $j=1$ we have a Pl\"{u}cker relation
\begin{displaymath}
    \ytableausetup{mathmode,boxsize=1em,centertableaux}
    \setlength{\delimitershortfall}{-5pt}
    \begin{ytableau}
        2 & 1 \\
        3 & 4 \\
        4 \\
        \end{ytableau}\,\, -
        \left(\begin{ytableau}
        1 & 2 \\
        4 & 3 \\
        4 \\
        \end{ytableau} +
        \begin{ytableau}
        1 & 2 \\
        3 & 4 \\
        4 \\
        \end{ytableau} +
        \begin{ytableau}
        2 & 3 \\
        1 & 4 \\
        4 \\
        \end{ytableau}\right).
\end{displaymath}
\end{example}

\begin{remark}
    When we quotient by the $R$-submodule $A(\lambda,z)$ the generators in (i) are redundant so long as $2$ is invertible in $R$; two equal entries in a non-cardinal filling $E$ may always be swapped, and so (ii) yields $E = \frac{1}{2}(E+E)=0$.
\end{remark}
 
\begin{theorem}[{\cite[\S 8.1]{MR1464693}}]
\label{theorem:fultonSSYTbasis}
The cosets of semistandard tableaux form a basis of the quotient module $R^{\fillcontent{\lambda}{z}} / A(\lambda,z)$ for every partition $\lambda$.
\end{theorem}

Our constructions differ from the ones found in \cite{MR1464693}. The first distinction is that our $R$-module $R^{\fillcontent{\lambda}{z}} / A(\lambda,z)$ is a submodule of the $R$-module $E^{\lambda}$ of \cite[\S 8.1]{MR1464693}. In particular, $$E^{\lambda} \cong \displaystyle \bigoplus_{|z|=|\lambda|} R^{\fillcontent{\lambda}{z}}/A(\lambda,z),$$ where $z$ ranges over all $n$-tuples of non-negative integers satisfying $|z|=|\lambda|$. 

The second distinction is that our Pl\"{u}cker relations are a subset of the ones found in \cite[\S 8.1, Lemma 1(iii)]{MR1464693}. The differences in \cite[\S 8.1, Lemma 1(iii)]{MR1464693} are of the form $E - \sum_{F} F$, where $E \in \fillcontent{\lambda}{z}$ and for two given columns and a positive integer $m$, the sum is over all fillings $F$ that result from $E$ by exchanging any $m$ elements of the leftmost given column with a fixed choice of $m$ elements from the rightmost given column (maintaining their vertical order). 

To reconcile these differences, fix a partition $\lambda$ and define 
\begin{center}
$F(\lambda) = \displaystyle \bigsqcup_{|z|=|\lambda|} F(\lambda,z)$ \qquad \text{and} \qquad $A(\lambda) = \displaystyle \bigoplus_{|z|=|\lambda|} A(\lambda,z)$,
\end{center}
where in both cases $z$ ranges over all $n$-tuples of non-negative integers satisfying $|z|=|\lambda|$, and the direct
sum defining $A(\lambda)$ is being taken inside of $R^{F(\lambda)} = \bigoplus_{|z| = |\lambda|} R^{F(\lambda, z)}$.
 Then 
\begin{center}
$R^{F(\lambda)} / A(\lambda) = \displaystyle \bigoplus_{|z|=|\lambda|} R^{F(\lambda,z)} / A(\lambda,z)$.
\end{center}

Theorem \ref{theorem:fultonSSYTbasis} is equivalent to the claim that the cosets of semistandard tableaux form a basis of $R^{F(\lambda)} / A(\lambda)$; such a basis of $R^{F(\lambda)} / A(\lambda)$ can be split apart into a basis of each addend $R^{F(\lambda,z)} / A(\lambda,z)$. Let $m = n$ to match the notation of \cite[\S 8.1, Lemma 1]{MR1464693}. Then $R^{F(\lambda)}$ is the $F$ from \cite[\S 8.1, Lemma 1]{MR1464693}. If we can show that $A(\lambda)$ is the $Q$ from \cite[\S 8.1, Lemma 1]{MR1464693}, then combining \cite[\S 8.1, Lemma 1]{MR1464693} and \cite[\S 8.1, Theorem 1]{MR1464693} proves Theorem \ref{theorem:fultonSSYTbasis}. To complete our proof we show $A(\lambda)$ is the $Q$ from \cite[\S 8.1, Lemma 1]{MR1464693}.

As remarked above, our Pl\"{u}cker relations are a subset of the ones found in \cite[\S 8.1, Lemma 1(iii)]{MR1464693}, and so $A(\lambda) \subseteq Q$. Thus $F/Q$ is a quotient of $F/A(\lambda)$. Now \cite[\S 8.1, Lemma 1]{MR1464693} identifies $F/Q$ with the $R$-module $E^{\lambda}$, and hence $E^{\lambda}$ is a quotient of $F/A(\lambda)$. We conclude, via \cite[\S 8.1, Lemma 3]{MR1464693}, that there exists a canonical map from $F/A(\lambda)$ to $R[Z]$ that maps each $T$ to $D_T$ (see Section \ref{subsec:RLinear} for the definitions of $R[Z]$ and $D_T$). 

The proof of \cite[\S 8.1, Theorem 1]{MR1464693} does not use all of the Pl\"{u}cker relations that generate $Q$. It uses only those Pl\"{u}cker relations where the two columns being exchanged are adjacent to each other and where the elements being swapped from the rightmost column are the top $m$ elements. These are precisely our Pl\"{u}cker relations, and so the proof given for \cite[\S 8.1, Theorem 1]{MR1464693} would work just as well for $F/A(\lambda)$ as for $F/Q$. Thus we conclude that the $R$-module $F/A(\lambda)$ also has a basis consisting of the semistandard tableaux. This completes the proof of Theorem \ref{theorem:fultonSSYTbasis}. Indeed, since the quotients $F/A(\lambda)$ and $F/Q$ have the same basis this implies the canonical projection from $F/A(\lambda)$ to $F/Q$ is an isomorphism. We conclude that $A(\lambda)=Q$.

%


\subsection{Straightening}
\label{subsec:straightening}
Let $F \in \fillcontent{\lambda}{z}$. When referring to the coset $[F] \in R^{\fillcontent{\lambda}{z}}/A(\lambda,z)$, we will abuse notation and simply write $F$. Expressing a filling as a $\mathbb{Z}$-linear combination of semistandard tableaux in the quotient $R^{\fillcontent{\lambda}{z}}/A(\lambda,z)$ is referred to as \emph{straightening} the filling. Giving a non-iterative formula for this straightening and the resulting coefficients is the primary goal of this paper.

All previously existing straightening algorithms that the author is aware of work in the following way. They prescribe a choice of a relation in $ A(\lambda, z)$ that allows a given filling to be rewritten as a sum of other fillings that are all smaller in some total order, and then proceed inductively (\cite{MR1035487}, \cite[\S 2.6]{MR1824028}, \cite[\S 3.1]{MR2667486}, \cite[\S 15.5]{MR1153249}). As there are a finite number of fillings of a given shape and content, these methods are guaranteed to terminate after a finite number of steps.

\subsection{Irreducible Representations}
\label{subsec:irreps}
In this subsection $R=\mathbb{C}$. The quotient module described above is important for a number of reasons. The partitions $\lambda$ with $|\lambda|=d$ index the irreducible representations of the symmetric group $S_d$. For a permutation $\sigma \in S_d$ and $F \in \fillcontent{\lambda}{(n \times 1)}$, $\sigma$ acts on $F$ by setting values equal to $i$ to $\sigma(i)$. This induces an action of $S_d$  on $\mathbb{C}^{\fillcontent{\lambda}{(n \times 1)}} / A(\lambda,(n \times 1))$, and the irreducible $S_d$-representation $[\lambda]$ is defined as $\mathbb{C}^{\fillcontent{\lambda}{(n \times 1)}} / A(\lambda,(n \times 1))$ with this action. Thus the standard tableaux of shape $\lambda$ give a basis of $[\lambda]$ (see \cite[Theorem 2.6.5]{MR1824028} or \cite[\S 7.4, Corollary]{MR1464693}).

In a similar fashion, the partitions $\lambda$ such that $l(\lambda) \leq n$ index the polynomial irreducible $\GL_n$-representations. Let $F(\lambda)$ be the finite set of fillings of shape $\lambda$ with entries in $[n]$ and denote by $\mathbb{C}^{F(\lambda)}$ the complex vector space with basis $F(\lambda)$. Define 

\begin{center}
$A(\lambda) = \displaystyle \bigoplus_{|z|=|\lambda|} A(\lambda,z)$
\end{center}

\noindent where $z$ ranges over all $n$-tuples of non-negative integers satisfying $|z|=|\lambda|$. Then there is an action of $\GL_n$ on $\mathbb{C}^{F(\lambda)} / A(\lambda)$ (see \cite[Chapter 4.1]{phdthesisCIkenmeyer}), and the vector space $\mathbb{C}^{F(\lambda)} / A(\lambda)$ along with this action defines the irreducible $\GL_n$-representation $\{\lambda\}$. The semistandard tableaux of shape $\lambda$ (and content $z$ such that $|\lambda|=|z|$) give a vector space basis of $\{ \lambda \}$ (see \cite[\S 8.1]{MR1464693}).

\section{Rearrangement Coefficients}
\label{sec:RCoeff}
We begin this section by defining the rearrangement coefficients and the $D$-basis. We then prove that for every $S \in \stdcontent{\lambda}{z}$, there exists a $R$-module homomorphism $$\reval{S}:R^{\fillcontent{\lambda}{z}} / A(\lambda, z) \rightarrow R$$ that maps $F \in F(\lambda, z)$ to the rearrangement coefficient associated to $F$ and $S$.
\subsection{Rearrangement Coefficients and the D-Basis}
\label{subsec:reArrCoeffDBasis}
\numberwithin{equation}{subsection}

Fix $n\geq2$ and $k\leq n$. Let $\lambda = (\lambda_1,\ldots,\lambda_k)$ be a partition, and $\zeta = (\zeta_1,\ldots,\zeta_{\lambda_1})$ its conjugate partition.

Define $\ycs{\lambda}{}$ to be the group of sequences of permutations $\underline{\pi}=(\pi_1,\ldots,\pi_{\lambda_1})$, where each $\pi_{i} \in S_{\zeta_i}$. We will refer to elements of $\ycs{\lambda}{}$ as \emph{multipermutations}. Given $\underline{\pi}, \underline{\sigma} \in \ycs{\lambda}{}$ with $\underline{\pi}=(\pi_1,\ldots,\pi_{\lambda_1})$ and $\underline{\sigma}=(\sigma_1,\ldots,\sigma_{\lambda_1})$, we have the composition $\underline{\pi} \circ \underline{\sigma} = (\pi_1 \circ \sigma_1,\ldots,\pi_{\lambda_1} \circ \sigma_{\lambda_1})\in \ycs{\lambda}{}$. Further, each element $\underline{\pi}=(\pi_1,\ldots,\pi_{\lambda_1}) \in \ycs{\lambda}{}$ has an inverse $\underline{\pi}^{-1}=((\pi_1)^{-1},\ldots,(\pi_{\lambda_1})^{-1})\in \ycs{\lambda}{}$.

If $F$ is a filling of shape $\lambda$ and $\underline{\pi}=(\pi_1,\ldots,\pi_{\lambda_1}) \in \ycs{\lambda}{}$, then we define $F_{\underline{\pi}}$ to be the filling that is obtained by permuting the values of $F$ in each column $i$ by $\pi_i$. Explicitly, for column index $1 \leq c \leq \lambda_1$ and row index $1 \leq r \leq \zeta_c$ we set
\begin{center}
$F_{\underline{\pi}}(r,c)=F(\perm{\pi^{-1}_{c}}{r},c)$.
\end{center}
Then for $\underline{\pi},\underline{\sigma} \in \ycs{\lambda}{}$ we have 
\[
    (F_{\underline{\pi}})_{\underline{\sigma}} = F_{\underline{\sigma}\underline{\pi}}.
\]

We define the sign of $\underline{\pi}$ to be
\begin{center}
$\sign{\underline{\pi}} \coloneqq \sign{\pi_1}\cdot \sign{\pi_2} \cdots \sign{\pi_{\lambda_1}} \in R$,
\end{center}
where for each positive integer $d$, $\textrm{sgn}:S_d \rightarrow R$ is the map that sends even permutations to $1_R$ and odd permutations to $-1_R$. It is an easy check that for $\underline{\pi}, \underline{\sigma} \in \ycs{\lambda}{}$ we have $\sign{\underline{\pi}\circ \underline{\sigma}}=\sign{\underline{\pi}}\cdot \sign{\underline{\sigma}}$ and $\sign{\underline{\pi}^{-1}}=\sign{\underline{\pi}}$. In addition we have that $F=\sign{\underline{\pi}} \cdot F_{\underline{\pi}}$ in the quotient module $R^{\fillcontent{\lambda}{z}} / A(\lambda, z)$. 

We define the \emph{rearrangement subset} of $\ycs{\lambda}{}$ associated to $F,S \in \fillcontent{\lambda}{z}$ to be the set
\begin{center}
$\ycs{\lambda}{F,S} = \{ \underline{\pi}\in \ycs{\lambda}{} \mid  F_{\underline{\pi}}\textrm{ and }S \textrm{ have the same row content} \}$.
\end{center}
This definition is not symmetric with respect to $F$ and $S$, and $\ycs{\lambda}{F,S}$ will rarely equal $\ycs{\lambda}{S,F}$.

\begin{definition} Let $F, S \in \fillcontent{\lambda}{z}$. We define the \emph{rearrangement coefficient} associated to these fillings to be
\begin{center}
$\rcf{F}{S}\coloneqq \displaystyle \sum_{\underline{\pi}\in \ycs{\lambda}{F,S}} \sign{\underline{\pi}}$
\end{center}
\end{definition}

\begin{example}
Let $\lambda = (4, 2, 2)$ and $z=(2,2,2,2)$ with $F,S \in \fillcontent{\lambda}{z}$ such that
\begin{center}
    \ytableausetup{mathmode,boxsize=1em,centertableaux}
$F=\begin{ytableau}
2 & 1 & 4 & 1 \\
3 & 2 \\
4 & 3 \\
\end{ytableau}
\qquad \textrm{and} \qquad
S=\begin{ytableau}
1 & 1 & 4 & 4 \\
2 & 2 \\
3 & 3 \\
\end{ytableau}\,\,$.
\end{center}
If we write our permutations in one-line notation, then $\underline{\pi} = (231,123,1,1)$ is in $\ycs{\lambda}{F,S}$ since
\begin{center}
\ytableausetup{mathmode,boxsize=1em,centertableaux}
$F_{\underline{\pi}}=\begin{ytableau}
4 & 1 & 4 & 1 \\
2 & 2 \\
3 & 3 \\
\end{ytableau}$
\end{center}
has the same row content as $S$. It is not difficult to check that $\underline{\pi}$ is the only element in $\ycs{\lambda}{F,S}$ and hence $\rcf{F}{S}=\sign{\underline{\pi}}=\sign{231} \cdot \sign{123} \cdot \sign{1} \cdot \sign{1} = 1$. The set $\ycs{\lambda}{S,F}$ is empty since there are no multipermutations that would rearrange the values within the columns of $S$ so that it has the same row content as $F$. Thus $\rcf{S}{F}=0$.
\end{example}

\begin{lemma}
\label{lemma:fundRCoeff}
Let $F, T, S \in \fillcontent{\lambda}{z}$. Let $\underline{\sigma}, \underline{\gamma}, \underline{\sigma}', \underline{\gamma}'$ be fixed elements of $\ycs{\lambda}{}$. Suppose that for all $\underline{\pi} \in \ycs{\lambda}{T,S}$ we have $\underline{\sigma} \circ \underline{\pi} \circ \underline{\gamma} \in \ycs{\lambda}{F,S}$. Further, suppose for all $\underline{\pi}' \in \ycs{\lambda}{F,S}$ we have $\underline{\sigma}' \circ \underline{\pi}' \circ \underline{\gamma}' \in \ycs{\lambda}{T,S}$. Then $\rcf{F}{S}=\sign{\underline{\sigma}}\cdot\sign{\underline{\gamma}}\cdot \rcf{T}{S} = \sign{\underline{\sigma}'}\cdot\sign{\underline{\gamma}'}\cdot \rcf{T}{S}$.
\end{lemma}
\begin{proof}
Consider the map $\Phi:\ycs{\lambda}{T,S} \longrightarrow \ycs{\lambda}{F,S}$ given by $\Phi(\underline{\pi})=\underline{\sigma} \circ \underline{\pi} \circ \underline{\gamma}$. This map is injective, since $\underline{\sigma} \circ \underline{\pi}_1 \circ \underline{\gamma} = \underline{\sigma} \circ \underline{\pi}_2 \circ \underline{\gamma}$ implies $\underline{\pi}_1 = \underline{\pi}_2$. In the same way, the map $\Psi:\ycs{\lambda}{F,S} \longrightarrow \ycs{\lambda}{T,S}$ given by $\Psi(\underline{\pi}')=\underline{\sigma}' \circ \underline{\pi}' \circ \underline{\gamma}'$ is injective. Thus, since both these sets are finite, we have that both $\Phi$ and $\Psi$ are bijective maps. Then
\begin{center}
$\displaystyle \rcf{F}{S}=\sum_{\underline{\pi}\in \ycs{\lambda}{F,S}} \sign{\underline{\pi}}=\sum_{\underline{\pi}\in \ycs{\lambda}{T,S}} \sign{\Phi(\underline{\pi})} = \sign{\underline{\sigma}}\cdot\sign{\underline{\gamma}}\cdot \rcf{T}{S}$
\end{center}
and
\begin{center}
$\displaystyle \rcf{T}{S}=\sum_{\underline{\pi}\in \ycs{\lambda}{T,S}} \sign{\underline{\pi}}=\sum_{\underline{\pi}\in \ycs{\lambda}{F,S}} \sign{\Psi(\underline{\pi})} = \sign{\underline{\sigma}'}\cdot\sign{\underline{\gamma}'}\cdot \rcf{F}{S}\,.$
\end{center}
The first of these two identities, coupled with the multiplication of both sides by $\sign{\underline{\sigma}'}\cdot\sign{\underline{\gamma}'}$ in the latter, gives us our  desired result.
\end{proof}
\begin{corollary}
\label{corollary:fundRCoeff}
Let $T, S \in \fillcontent{\lambda}{z}$. Then for $\underline{\pi} \in \ycs{\lambda}{}$ we have $\rcf{T_{\underline{\pi}}}{S}=\sign{\underline{\pi}}\cdot \rcf{T}{S}$.
\end{corollary}

We now prove a technical lemma and a proposition that illustrate the relationship between the reading word order and rearrangement coefficients.

\begin{lemma}
\label{lemma:reorderIncreases}
Let $T \in \tabcontent{\lambda}{z}$ and $\underline{\pi} \in \ycs{\lambda}{}$ such that $\underline{\pi} \neq (\operatorname{id},\ldots,\operatorname{id})$. If $S = \rowsorting{T_{\underline{\pi}}}$, then $S \succ \sorting{T}$.
\end{lemma}
\begin{proof}
Let $r$ be the index of the first row where $T_{\underline{\pi}}$ and $T$ differ; such a row must exist since $\underline{\pi} \neq (\operatorname{id},\ldots,\operatorname{id})$. Let $c$ be the index of the first column where $T_{\underline{\pi}}$ and $T$ differ in row $r$. Then $T_{\underline{\pi}}(r,c)>T(r,c)$ because the columns of $T$ are strictly increasing downward. For any $q \leq \lambda_r$ we must also have $T_{\underline{\pi}}(r,q)\geq T(r,q)$, again because the columns of $T$ are strictly increasing downward, and row $r$ is the first row where $T_{\underline{\pi}}$ and $T$ differ. 

As $T_{\underline{\pi}}$ and $T$ are the same for rows less than $r$, so are $S$ and $\sorting{T}=\rowsorting{T}$. In row $r$, the values in each location in $T_{\underline{\pi}}$ are all greater than or equal the value in the corresponding location in $T$, with at least one being strictly greater. Hence after reordering within the rows of $T_{\underline{\pi}}$ and $T$ such that they are non-decreasing to get $S$ and $\sorting{T}$ respectively, we must have, by the ``Non-Messing-Up'' Theorem\cite{MR2311934}, that the entries in row $r$ of $S$ are lexicographically greater than the entries in row $r$ of $\sorting{T}$. Thus $S \succ \sorting{T}$.
\end{proof}

\begin{proposition}
\label{proposition:rCoeffOrder}
Let $F \in \fillcontentc{\lambda}{z}$ be a cardinal filling. Let $T \in \tabcontent{\lambda}{z}$ be the tableau that results from ordering the values within the columns of $F$ so that they increase strictly downwards. Then $T = F_{\underline{\sigma}}$ for some $\underline{\sigma} \in \ycs{\lambda}{}$. Let $S \in \stdcontent{\lambda}{z}$. Then
\begin{enumerate}[label=(\roman*)]
\item $S \prec \sorting{F}$ implies $\rcf{F}{S}=0$.
\item $F \in \tabcontent{\lambda}{z}$ implies $\rcf{F}{\sorting{F}}=1$.
\item $\rcf{F}{\sorting{F}}=\sign{\underline{\sigma}}$.
\end{enumerate}
\end{proposition}
\begin{proof}
\noindent $\textit{(i)}$ By Corollary \ref{corollary:fundRCoeff} we have $\rcf{T}{S} = \sign{\underline{\sigma}} \cdot \rcf{F}{S}$. We will show that $S \prec \sorting{F} = \sorting{T}$ implies $\rcf{T}{S}=0$, and hence $\rcf{F}{S}=0$. Let $S \prec \sorting{T}$. Now suppose that for some $\underline{\pi}\in \ycs{\lambda}{}$ we have that $T_{\underline{\pi}}$ has the same row content as $S$. Then $S = \rowsorting{T_{\underline{\pi}}}$. 

If $\underline{\pi}=(\operatorname{id},\ldots,\operatorname{id})$, then $T=T_{\underline{\pi}}$. Hence $S=\rowsorting{T}$, and since the columns of $T$ are strictly increasing downward this means $S  = \sorting{T}$. This contradicts our initial assumption. If $\underline{\pi}\neq(\operatorname{id},\ldots,\operatorname{id})$, then since $S = \rowsorting{T_{\underline{\pi}}}$ we have by Lemma~\ref{lemma:reorderIncreases} that $S \succ \sorting{T}$. Hence our initial assumption is contradicted in this case as well. Thus there are no $\underline{\pi}\in \ycs{\lambda}{}$ such that $T_{\underline{\pi}}$ has the same row content as $S$, and so $\ycs{\lambda}{T, S} = \emptyset$ and $\rcf{T}{S}=0$.

\noindent $\textit{(ii)}$ If $F \in \tabcontent{\lambda}{z}$ then $F=T$. If $\underline{\pi}=(\operatorname{id},\ldots,\operatorname{id})\in \ycs{\lambda}{}$, then $T_{\underline{\pi}}$ has the same row content as $\sorting{T}$ (since $T_{\underline{\pi}}=T$ and $\sorting{T}=\rowsorting{T}$). For any other $\underline{\pi}'\in \ycs{\lambda}{}$, suppose that $T_{\underline{\pi}'}$ has the same row content as $\sorting{T}$. This implies that $\rowsorting{T_{\underline{\pi}'}} = \sorting{T}$. Applying Lemma~\ref{lemma:reorderIncreases} we conclude that $\sorting{T} \succ \sorting{T}$. As this is clearly not possible, we must have that $\underline{\pi}$ is the only element in $\ycs{\lambda}{T, \sorting{T}}$, and hence  $\rcf{F}{\sorting{F}}=\rcf{T}{\sorting{T}}=\sign{\underline{\pi}}=1$.

\noindent $\textit{(iii)}$ This follows from part $\textit{(ii)}$, Corollary \ref{corollary:fundRCoeff}, and $\sorting{F}=\sorting{T}$.
\end{proof}

We label all the elements in $\stdcontent{\lambda}{z}$ such that
\begin{equation}
\label{equation:rowOrderStd}
S_{1} \succ S_{2} \succ \cdots \succ S_{\kostka{\lambda}{z}}.
\end{equation}
This is well defined since $\succ$ is a total order.
\begin{definition}
For each $S_{i} \in \stdcontent{\lambda}{z}$ we define
\begin{center}
$\rbasis{S_{i}} := S_{i} - \displaystyle \sum_{\substack{S_j \in \stdcontent{\lambda}{z} \\ \textrm{ such that } j < i}} \rcf{S_{i}}{S_j} \cdot \rbasis{S_j}.$
\end{center}
It is not difficult to see that the set $\{ \rbasis{S_{i}} \mid 1 \leq i \leq  \kostka{\lambda}{z} \}$ is a basis for $R^{\fillcontent{\lambda}{z}} / A(\lambda, z)$; this follows by noting that the transition matrix from the basis $\stdcontent{\lambda}{z}$ of $R^{\fillcontent{\lambda}{z}} / A(\lambda, z)$ is triangular and has $1$'s on the diagonal. We will refer to $\{ \rbasis{S_{i}} \mid 1 \leq i \leq  \kostka{\lambda}{z} \}$ as the \emph{D-basis} for $R^{\fillcontent{\lambda}{z}} / A(\lambda, z)$.
\end{definition}

\subsection{A R-linear map}
\label{subsec:RLinear}

Let $S\in \stdcontent{\lambda}{z}$. We claim that there exists a $R$-module homomorphism $\reval{S}:R^{\fillcontent{\lambda}{z}} / A(\lambda, z) \rightarrow R$ that maps $F \in F(\lambda, z)$ to $\rcf{F}{S}$. The homomorphism $\reval{S}$ will be defined as the composition of two $R$-module homomorphisms that we now introduce.

Let $m=n$. Let $Z_{i,j}$ for $i \in [n], j \in [m]$ be $nm$ indeterminates. Define $R[Z]$ to be the polynomial ring over $R$ in these $nm$ indeterminates. For a $p$-tuple $(i_1,\ldots,i_p) \in [m]^p$ with $p \leq m$ we define 
\[
D_{i_1,\ldots,i_p} = \det \begin{bmatrix} 
Z_{1,i_1} & \dots & Z_{1,i_p} \\
\vdots & \, & \vdots \\
Z_{p,i_1} &  \dots     & Z_{p,i_p} 
\end{bmatrix} \in R[Z].
\]
For $F \in \fillcontent{\lambda}{z}$ set
\[
D_F = \displaystyle \prod_{j=1}^{\lambda_1} D_{F(1, j), F(2, j), \ldots, F(\zeta_j, j)}\,.
\]
By \cite[\S 8.1, Lemma 3]{MR1464693} there is a $R$-module homomorphism $\varphi$ from $E^{\lambda}$ to $R[Z]$ that maps $F\!\in\!F(\lambda, z)$ to $D_F$. This restricts to a $R$-module homomorphism from $R^{\fillcontent{\lambda}{z}} / A(\lambda,z)$ to $R[Z]$.

Let  $\mathfrak{C}_{-,S}$ be the map from $R[Z]$ to $R$ that sends $p \in R[Z]$ to the coefficient of the monomial $$\displaystyle \prod_{j=1}^{\lambda_1} \prod_{i=1}^{\zeta_j} Z_{i, S(i,j)}$$ in $p$. This map is trivially a $R$-module homomorphism.

Finally, set 
\[
    \reval{S}:=\mathfrak{C}_{-,S} \circ \varphi.
\]

\begin{proposition}
\label{proposition:theRModuleHom}
Let $S\in \stdcontent{\lambda}{z}$. The map $\reval{S}$ is a $R$-module homomorphism from $R^{\fillcontent{\lambda}{z}} / A(\lambda, z)$ to $R$ that maps $F \in F(\lambda, z)$ to $\rcf{F}{S}$.
\end{proposition}
\begin{proof}
That $\reval{S}$ is a $R$-module homomorphism follows immediately from the fact it is defined as the composition of two $R$-module homomorphisms. It remains to show that $\reval{S}(F) = \rcf{F}{S}$, or equivalently that $\mathfrak{C}_{-,S}(D_F) = \rcf{F}{S}$. We have
\begin{align*}
       D_F &= \displaystyle \prod_{j=1}^{\lambda_1} D_{F(1, j), F(2, j), \ldots, F(\zeta_j, j)} \\
       \,  &= \displaystyle \prod_{j=1}^{\lambda_1} \sum_{\sigma \in S_{\zeta_j}} \sign{\sigma} Z_{1,F(\sigma(1), j)} Z_{2, F(\sigma(2), j)} \cdots Z_{\zeta_j, F(\sigma(\zeta_j), j)}  \\
       \,  &= \displaystyle \prod_{j=1}^{\lambda_1} \sum_{\sigma \in S_{\zeta_j}} \sign{\sigma} Z_{1,F(\sigma^{-1}(1), j)} Z_{2, F(\sigma^{-1}(2), j)} \cdots Z_{\zeta_j, F(\sigma^{-1}(\zeta_j), j)} \, , \\
       \intertext{where the second equality is expanding the determinant, and the third equality is a re-indexing of the summands under the bijection from $S_{\zeta_j}$ to itself sending each element to its inverse and subsequently applying $\sign{\sigma} = \sign{\sigma^{-1}}$. Continuing the above equation }
       \,  &= \displaystyle \sum_{\underline{\pi}=(\pi_1,\ldots,\pi_{\lambda_1}) \in \ycs{\lambda}{}} \sign{\underline{\pi}} \prod_{j=1}^{\lambda_1} Z_{1,F(\pi_j^{-1}(1), j)} Z_{2, F(\pi_j^{-1}(2), j)} \cdots Z_{\zeta_j, F(\pi_j^{-1}(\zeta_j), j)} \\
       \,  &= \displaystyle \sum_{\underline{\pi} \in \ycs{\lambda}{}} \sign{\underline{\pi}} \prod_{j=1}^{\lambda_1} \prod_{i=1}^{\zeta_j} Z_{i,F_{\underline{\pi}}(i, j)} \\ \,  &= \displaystyle \sum_{\underline{\pi} \in \ycs{\lambda}{}} \sign{\underline{\pi}} \prod_{i=1}^{\zeta_1} \prod_{j=1}^{\lambda_i} Z_{i,F_{\underline{\pi}}(i, j)}\, .\\
\end{align*}
A summand from the above equation will contain the monomial $$\displaystyle \prod_{j=1}^{\lambda_1} \prod_{i=1}^{\zeta_j} Z_{i, S(i,j)} = \displaystyle \prod_{i=1}^{\zeta_1} \prod_{j=1}^{\lambda_i} Z_{i, S(i,j)}$$ precisely when for every fixed $i \in [\zeta_1]$ the multiset $\{ S(i,j) | j \in [\lambda_i] \}$ equals the multiset $\{ F_{\underline{\pi}}(i,j) | j \in [\lambda_i] \}$. This occurs if and only if $F_{\underline{\pi}}$ has the same row content as $S$; that is, $\underline{\pi} \in \ycs{\lambda}{F,S}$. We conclude that $$\mathfrak{C}_{-,S}(D_F) = \displaystyle \sum_{\underline{\pi} \in \ycs{\lambda}{F,S}} \sign{\underline{\pi}} = \rcf{F}{S}\,.$$ \end{proof}  

\section{Main Results}
\label{sec:Main}
\subsection{Straightening Formula}
\label{subsec:mainResults}
We are now ready to prove our primary results. This first lemma proves Theorem \ref{theorem:mainTheoremStr1In} in the case when $F$ is a semistandard tableau.

\begin{lemma}
\label{lemma:mainTheorem1Standard}
Let $S_i \in \stdcontent{\lambda}{z}$. Then 
\begin{center}
$\displaystyle S_i = \displaystyle \sum_{S_j \in \stdcontent{\lambda}{z}} \rcf{S_i}{S_j} \cdot \rbasis{S_j}$.
\end{center}
\end{lemma}
\begin{proof}
For each $j > i$ we have $S_j \prec S_i = \sorting{S_i}$ and thus $\rcf{S_i}{S_j}=0$ by Proposition~\ref{proposition:rCoeffOrder}(i). Thus 
\begin{center}
$\begin{array}{rlr}
\displaystyle \sum_{S_j \in \stdcontent{\lambda}{z}} \rcf{S_i}{S_j} \cdot \rbasis{S_j} & = \displaystyle \sum_{\substack{S_j \in \stdcontent{\lambda}{z} \\ \textrm{such that } j \leq i}} \rcf{S_{i}}{S_j} \cdot \rbasis{S_j}  \\
\; & \; \; \\
\; & = \rbasis{S_i} + \displaystyle \sum_{\substack{S_j \in \stdcontent{\lambda}{z} \\ \textrm{such that } j < i}} \rcf{S_{i}}{S_j} \cdot \rbasis{S_j} \qquad \qquad \textrm{by Prop }\ref{proposition:rCoeffOrder}(ii) \\
\; & \; \; \\
\; & = \left(S_{i} - \displaystyle \sum_{\substack{S_j \in \stdcontent{\lambda}{z} \\ \textrm{such that } j < i}} \rcf{S_{i}}{S_j} \cdot \rbasis{S_j}\right) + \displaystyle \sum_{\substack{S_j \in \stdcontent{\lambda}{z} \\ \textrm{such that } j < i}} \rcf{S_{i}}{S_j} \cdot \rbasis{S_j} \\
\; & \; \; \\
\; & = S_i \\
\end{array}$
\end{center}
\end{proof}

\begin{proposition}
\label{proposition:mainStraighten}
Suppose that $F \in \fillcontent{\lambda}{z}$, with $F=\sum_ {S_i \in \stdcontent{\lambda}{z}}  a_i S_i$ in $R^{\fillcontent{\lambda}{z}} / A(\lambda, z)$ (where $a_i \in R$). Then
\begin{center}
$\rcf{F}{S_j} = \displaystyle \sum_{S_i \in \stdcontent{\lambda}{z}} a_i \cdot \rcf{S_i}{S_j}$
\end{center}
for each $S_j \in \stdcontent{\lambda}{z}$.
\end{proposition}
\begin{proof}
By Proposition~\ref{proposition:theRModuleHom}, applying the $R$-module homomorphism $\reval{S_j}$ to both sides of the equation $F=\sum_ {S_i \in \stdcontent{\lambda}{z}}  a_i S_i$ yields the desired equality for each $S_j \in \stdcontent{\lambda}{z}$.
\end{proof}

We can now prove Theorem 1.1, which we restate here for convenience.

\newtheorem*{theorem:mainTheoremStr1In}{Theorem \ref*{theorem:mainTheoremStr1In}}
\begin{theorem:mainTheoremStr1In}
\textit{Let $F \in \fillcontent{\lambda}{z}$, with $F=\sum_ {S_i \in \stdcontent{\lambda}{z}}  a_i S_i$ in $R^{\fillcontent{\lambda}{z}} / A(\lambda, z)$. Then
\begin{center}
$\displaystyle \sum_{S_i \in \stdcontent{\lambda}{z}} a_i S_i = \displaystyle \sum_{S_j \in \stdcontent{\lambda}{z}} \rcf{F}{S_j} \cdot \rbasis{S_j}$.
\end{center}}
\end{theorem:mainTheoremStr1In}
\begin{proof}
By Proposition \ref{proposition:mainStraighten}
\begin{center}
$\rcf{F}{S_j} = \displaystyle \sum_{S_i \in \stdcontent{\lambda}{z}} a_i \cdot \rcf{S_i}{S_j}$\,
\end{center}
for each $S_j \in \stdcontent{\lambda}{z}$. This further implies that 
\begin{center}
$\left( \rcf{F}{S_j} - \displaystyle \sum_{S_i \in \stdcontent{\lambda}{z}} a_i \cdot \rcf{S_i}{S_j} \right) \cdot \rbasis{S_j} = 0$
\end{center}
for each $S_j \in \stdcontent{\lambda}{z}$. Summing over all such terms we have
\[
\displaystyle \sum_{S_j \in \stdcontent{\lambda}{z}} \left( \rcf{F}{S_j} - \displaystyle \sum_{S_i \in \stdcontent{\lambda}{z}} a_i \cdot \rcf{S_i}{S_j} \right) \cdot \rbasis{S_j} = 0\,.
\]
Rearranging the left hand side summands from the above equation results in
\[
\displaystyle \sum_{S_j \in \stdcontent{\lambda}{z}} \rcf{F}{S_j} \cdot \rbasis{S_j} - \displaystyle \sum_{S_j \in \stdcontent{\lambda}{z}} \sum_{S_i \in \stdcontent{\lambda}{z}} a_i \cdot \rcf{S_i}{S_j} \cdot \rbasis{S_j} = 0
\]
and 
\[
\displaystyle \sum_{S_j \in \stdcontent{\lambda}{z}} \rcf{F}{S_j} \cdot \rbasis{S_j} - \displaystyle \sum_{S_i \in \stdcontent{\lambda}{z}} a_i \cdot \sum_{S_j \in \stdcontent{\lambda}{z}} \rcf{S_i}{S_j} \cdot \rbasis{S_j} = 0. 
\]
Finally, applying Lemma~\ref{lemma:mainTheorem1Standard} to this equation yields
\[
\displaystyle \sum_{S_j \in \stdcontent{\lambda}{z}} \rcf{F}{S_j} \cdot \rbasis{S_j} - \displaystyle \sum_{S_i \in \stdcontent{\lambda}{z}} a_i \cdot S_i = 0 
\]
allowing us to conclude
\[
\displaystyle \sum_{S_i \in \stdcontent{\lambda}{z}} a_i S_i = \displaystyle \sum_{S_j \in \stdcontent{\lambda}{z}} \rcf{F}{S_j} \cdot \rbasis{S_j} \,.
\]
\end{proof}

\begin{corollary}
\label{corollary:mainCorollaryStr}
Suppose that $F \in \fillcontentc{\lambda}{z}$, with $\sorting{F}=S_k \in \stdcontent{\lambda}{z}$. If $F=\sum_ {S_i \in \stdcontent{\lambda}{z}}  a_i S_i$ in $R^{\fillcontent{\lambda}{z}} / A(\lambda, z)$, then
\begin{center}
$\displaystyle \sum_{S_i \in \stdcontent{\lambda}{z}} a_i S_i = \displaystyle \sum_{\substack{S_j \in \stdcontent{\lambda}{z} \\ \textrm{such that } j \leq k}} \rcf{F}{S_j} \cdot \rbasis{S_j}$
\end{center}
\end{corollary}
\begin{proof}
This follows by Lemma \ref{lemma:sortingSST} and Proposition \ref{proposition:rCoeffOrder}(i).
\end{proof}

Our second main result, Corollary~\ref*{corollary:mainCorollaryStr2}, follows as an immediate consequence.

\newtheorem*{corollary:mainCorollaryStr2}{Corollary \ref*{corollary:mainCorollaryStr2}}
\begin{corollary:mainCorollaryStr2}
\textit{Suppose that $F \in \fillcontent{\lambda}{z}$ is a cardinal filling, with $\sorting{F}=S_k \in \stdcontent{\lambda}{z}$. If $F=\sum_ {S_i \in \stdcontent{\lambda}{z}}  a_i S_i$ in $R^{\fillcontent{\lambda}{z}} / A(\lambda, z)$, then $a_k = \sign{\underline{\sigma}}$ and $a_l = 0$ for all $l > k$, where $\underline{\sigma}$ is the unique multipermutation in $C(\lambda)$ such that $F_{\underline{\sigma}}$ is a tableau. If $F$ is a tableau, $a_k = 1$.}
\end{corollary:mainCorollaryStr2}
\begin{proof}
Corollary \ref{corollary:mainCorollaryStr} and the definition of the $D$-basis imply that $a_l = 0$ for all $l > k$ and that $a_k = \rcf{F}{S_k}$. Then $\rcf{F}{S_k} = \rcf{F}{\sorting{F}}=\sign{\underline{\sigma}}$ by Proposition \ref{proposition:rCoeffOrder}(iii). Finally, if $F$ is a tableau, then $a_k = 1$ by Proposition \ref{proposition:rCoeffOrder}(ii).
\end{proof}

Corollary \ref{corollary:mainCorollaryStr} is a generalization of Theorem 5.2(1) from~\cite{MR1417711}, where Lakshmibai and Gonciulea prove this for two column shapes of the form $\lambda = p \times 2$, for some positive integer $p$. Their notation differs somewhat and their result is stated in terms of lattices, meets, and joins but the sorting of a tableau with a two column shape corresponds to the leading term whose coefficient they prove is equal to 1.

\begin{remark}
In general, calculating the left hand side of the equation in Theorem~\ref{theorem:mainTheoremStr1In} using a classical straightening algorithm is a slow, iterative process.  Computing the right hand side in Theorem~\ref{theorem:mainTheoremStr1In} can be achieved via two steps. In the first, the $D$-basis is computed. In the second, the filling $F$ may be straightened by calculating $\rcf{F}{S_j}$ for $S_j \in \stdcontent{\lambda}{z}$. In both steps, every rearrangement coefficient may be computed in parallel, making implementation on modern computer architecture considerably more efficient.

Additionally, the D-basis only depends on the content $z$ and shape $\lambda$, and hence, once it has been computed, straightening \emph{any filling} with the same content and shape reduces to the computation of (at most) $\kostka{\lambda}{z}$ rearrangement coefficients.
\end{remark}
 
\begin{example}
Let the partition $\lambda = (4,3,2)$ and the content $z=(2,2,3,2)$. The six semistandard tableaux with this shape and content, ordered as in \eqref{equation:rowOrderStd}, are
\begin{center}
$\ytableausetup{mathmode,boxsize=1em,centertableaux}
S_1=\begin{ytableau}
1 & 1 & 3 & 4 \\
2 & 2 & 4 \\
3 & 3 \\
\end{ytableau}
\qquad \qquad
S_2=\begin{ytableau}
1 & 1 & 3 & 3 \\
2 & 2 & 4 \\
3 & 4 \\
\end{ytableau}
\qquad \qquad
S_3=\begin{ytableau}
1 & 1 & 2 & 4 \\
2 & 3 & 3 \\
3 & 4 \\
\end{ytableau}$

~\\

$S_4=\begin{ytableau}
1 & 1 & 2 & 3 \\
2 & 3 & 4 \\
3 & 4 \\
\end{ytableau}
\qquad \qquad
S_5=\begin{ytableau}
1 & 1 & 2 & 3 \\
2 & 3 & 3 \\
4 & 4 \\
\end{ytableau}
\qquad \qquad
S_6=\begin{ytableau}
1 & 1 & 2 & 2 \\
3 & 3 & 3 \\
4 & 4 \\
\end{ytableau}$
\end{center}

We can now calculate the $D$-basis associated to this shape and content. We have
\begin{center}
$\begin{array}{rl}
\rbasis{S_{1}} & = S_{1} \\
\rbasis{S_{2}} & = S_{2} \\
\rbasis{S_{3}} & = S_{3} - \rbasis{S_{1}} \\
\; & = S_{3} - S_{1} \\
\rbasis{S_{4}} & = S_{4} - \rbasis{S_{1}} \\
\; & = S_{4} - S_{1} \\
\rbasis{S_{5}} & = S_{5} + \rbasis{S_{4}} - \rbasis{S_{2}} \\
\; & = S_{5} + S_{4} - S_{2} - S_{1}\\
\rbasis{S_{6}} & = S_{6} + \rbasis{S_{5}} - 2 \cdot \rbasis{S_{4}} \\
\; & = S_{6} + S_{5} - S_{4} - S_{2} + S_{1}\\
\end{array}$ 
\end{center}
Now, if we have a $F \in \fillcontent{\lambda}{z}$ such that 
\begin{center}
$\ytableausetup{mathmode,boxsize=1em,centertableaux}
F=\begin{ytableau}
2 & 1 & 1 & 3 \\
3 & 3 & 2 \\
4 & 4 \\
\end{ytableau}$
\end{center}
we may express it as a $\mathbb{Z}$-linear combination of semistandard tableaux in the quotient module $R^{\fillcontent{\lambda}{z}} / A(\lambda, z)$ using Theorem \ref{theorem:mainTheoremStr1In}. We have $\rcf{F}{S_6} = 0$, $\rcf{F}{S_5} = 1$, $\rcf{F}{S_4} = -2$, $\rcf{F}{S_3} = 0$, $\rcf{F}{S_2} = 1$, and $\rcf{F}{S_1} = -1$. Thus
\begin{center}
$\begin{array}{rl}
F &= \rbasis{S_{5}} - 2 \cdot \rbasis{S_{4}} + \rbasis{S_{2}} - \rbasis{S_{1}} \\
\; &= S_5 - S_4
\end{array}$  
\end{center}
\end{example}

\subsection{Straightening Coefficients}
\label{subsec:strCoeff}

We may also describe the coefficients $a_i$ that appear on the right hand side of the straightening of the filling $F =  \sum_ {S_i \in \stdcontent{\lambda}{z}} a_i S_i$ in $R^{\fillcontent{\lambda}{z}} / A(\lambda, z)$. We hope to use the results from this section to analyze the computational complexity of straightening in a subsequent work.

We now prove a proposition regarding the coefficients in the D-basis elements themselves and as a corollary get a formula for the coefficients in the straightening.

\begin{proposition}
\label{proposition:stdCoeffDBasis}
Let $S_i, S_j \in \stdcontent{\lambda}{z}$. Then $S_i$ shows up in $\rbasis{S_{j}}$ with coefficient equal to
\begin{equation}
\label{equation:coeffeq}
\displaystyle \sum_{\substack{(b_0,\ldots,b_d) \in \mathbb{N}_{>0}^{d+1} \\ \textrm{such that } i=b_0<b_1<\cdots<b_d=j \\ \textrm{and }d \geq 0}} (-1)^{d} \cdot \rcf{S_{b_d}}{S_{b_{d-1}}} \cdots \rcf{S_{b_1}}{S_{b_0}}\,.
\end{equation}
\end{proposition}
\begin{proof}
We have by definition that
\begin{equation}
\label{equation:dbasiseq}
\rbasis{S_{j}} = S_{j} - \displaystyle \sum_{\substack{S_k \in \stdcontent{\lambda}{z} \\ \textrm{such that } k < j}} \rcf{S_{j}}{S_k} \cdot \rbasis{S_k}\,.
\end{equation}
If $i=j$, then $d$ is always zero in \eqref{equation:coeffeq} and the only summand is $(-1)^d = (-1)^0 = 1$. As $S_j$ appears with coefficient $1$ in \eqref{equation:dbasiseq} the proposition holds. 

In the case where $i \neq j$, we proceed by strong induction, assuming that the proposition holds for any $k < j$.  By the induction hypothesis applied to the $\rbasis{S_k}$ in \eqref{equation:dbasiseq}, we have that $S_i$ will appear in \eqref{equation:dbasiseq} with coefficient equal to
\begin{center}
$\displaystyle (-1) \cdot \sum_{k < j} \rcf{S_{j}}{S_{k}} \sum_{\substack{(b_0,\ldots,b_d) \in \mathbb{N}_{>0}^{d+1} \\ \textrm{such that } i=b_0<b_1<\cdots<b_d=k \\ \textrm{and }d \geq 0}} (-1)^{d} \cdot \rcf{S_{b_d}}{S_{b_{d-1}}} \cdot \rcf{S_{b_{d-1}}}{S_{b_{d-2}}} \cdots \rcf{S_{b_1}}{S_{b_0}}\,.$
\end{center}
It is not difficult to check that this is precisely equal to \eqref{equation:coeffeq} and hence the proposition holds.
\end{proof}
\begin{corollary}
\label{corollary:stdCoeffStr}
Suppose that $F \in \fillcontentc{\lambda}{z}$, with $F=\sum_ {S_i \in \stdcontent{\lambda}{z}}  a_i S_i$ in $R^{\fillcontent{\lambda}{z}} / A(\lambda, z)$. Then $\sorting{F}=S_k \in \stdcontent{\lambda}{z}$ for some $1 \leq k \leq \kostka{\lambda}{z}$ and
\begin{equation}
\label{eq:summationFormulaStr}
a_i = \displaystyle \sum_{\substack{(b_0,\ldots,b_d) \in \mathbb{N}_{>0}^{d+1} \\ \textrm{such that } i=b_0<b_1<\cdots<b_d \leq k \\ \textrm{and }d \geq 0}} (-1)^{d} \cdot \rcf{F}{S_{b_d}} \cdot \rcf{S_{b_d}}{S_{b_{d-1}}} \cdots \rcf{S_{b_1}}{S_{b_0}}\,.
\end{equation}
\end{corollary}
\begin{proof}
The fact that $\sorting{F}=S_k \in \stdcontent{\lambda}{z}$ for some $1 \leq k \leq \kostka{\lambda}{z}$ is simply Lemma \ref{lemma:sortingSST}. The formula for the coefficient is immediate by the preceding proposition and Corollary \ref{corollary:mainCorollaryStr}.
\end{proof}

The summation in Corollary \ref{corollary:stdCoeffStr} can be used to compute individual coefficients but it is not particularly enlightening; the vast majority of summands are zero. To get a better understanding of these coefficients we may use a graph to visualize the nonzero summands from the corollary.


Fix a partition $\lambda$ and content $z$. Then, let $G=(V,E)$ be a directed graph, with vertex set $V$ and edge set $E \subseteq V \times V$. The vertex set $V$ is the set $\stdcontent{\lambda}{z}$ of semistandard tableaux of shape $\lambda$ and content $z$. The edge set $E$ will contain the edge $(S_i,S_j)$ if and only if $i \neq j$ and $\rcf{S_i}{S_j} \neq 0$. With these definitions we see by Proposition \ref{proposition:rCoeffOrder}(i) that there are no cycles in this directed graph.

\begin{example}
\label{example:graph}
Suppose $\lambda=(3,3,2)$ and $z=(1,2,1,2,2)$. The six semistandard tableaux with this shape and content, ordered as in \eqref{equation:rowOrderStd}, are
\begin{center}
$\ytableausetup{mathmode,boxsize=1em,centertableaux}
S_1=\begin{ytableau}
1 & 2 & 4 \\
2 & 4 & 5 \\
3 & 5 \\
\end{ytableau}
\qquad \qquad
S_2=\begin{ytableau}
1 & 2 & 4 \\
2 & 3 & 5 \\
4 & 5 \\
\end{ytableau}
\qquad \qquad
S_3=\begin{ytableau}
1 & 2 & 3 \\
2 & 4 & 5 \\
4 & 5 \\
\end{ytableau}$

~\\

$S_4=\begin{ytableau}
1 & 2 & 3 \\
2 & 4 & 4 \\
5 & 5 \\
\end{ytableau}
\qquad \qquad
S_5=\begin{ytableau}
1 & 2 & 2 \\
3 & 4 & 5 \\
4 & 5 \\
\end{ytableau}
\qquad \qquad
S_6=\begin{ytableau}
1 & 2 & 2 \\
3 & 4 & 4 \\
5 & 5 \\
\end{ytableau}$
\end{center}
The vertex set $V=\{ S_1, S_2, S_3, S_4, S_5, S_6 \}$ and after computing the rearrangement coefficients we get that the edge set 
\begin{center}
$E = \{ (S_6, S_5) , (S_6, S_2) , (S_6, S_1) , (S_5, S_2) , (S_5, S_1) , (S_4, S_3) , (S_4, S_2) \}$.
\end{center}
Thus the graph $G = (V, E)$ is
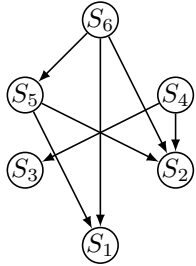
\begin {center}
\begin {tikzpicture}[-latex ,auto ,node distance =3 cm and 4cm ,on grid ,
semithick , state/.style ={ circle ,top color = white , bottom color = white,
draw, black , text=black , minimum width =0.5 cm, inner sep=0pt,minimum size=1pt}]
\node[state] at (0, 0)   (S1) {\small $S_1$};
\node[state] at (1, 1)   (S2) {\small $S_2$};
\node[state] at (-1, 1)   (S3) {\small $S_3$};
\node[state] at (1, 2)  (S4) {\small $S_4$};
\node[state] at (-1, 2)   (S5) {\small $S_5$};
\node[state] at (0, 3)   (S6) {\small $S_6$};

\draw[] (S6) -- (S5) node[at start, above right] {$ $};
\draw[] (S6) -- (S2) node[at start, above right] {$ $};
\draw[] (S6) -- (S1) node[at start, above right] {$ $};
\draw[] (S5) -- (S2) node[at start, above right] {$ $};
\draw[] (S5) -- (S1) node[at start, above right] {$ $};
\draw[] (S4) -- (S3) node[at start, above right] {$ $};
\draw[] (S4) -- (S2) node[at start, above right] {$ $};
\end{tikzpicture}
\end{center}
\end{example}
For $S_i, S_j \in V$ let $\mathcal{P}(S_i,S_j)$ be the set of all paths in $G$ starting at vertex $S_i$ and ending at $S_j$. We will denote paths by a list of vertices enclosed in brackets, that is $\langle S_{p_d},S_{p_{d-1}},\ldots,S_{p_1},S_{p_0} \rangle$ indicates there is an edge from vertex $S_{p_d}$ to $S_{p_{d-1}}$, $S_{p_{d-1}}$ to $S_{p_{d-2}}$, and so on in the graph $G$. Note that we allow paths that only contain a single vertex, so for example $\mathcal{P}(S_i,S_i)= \{ \langle S_i \rangle \}$.
Now we fix an $F \in \fillcontent{\lambda}{z}$. Let $V_F \subseteq V$ be the set of vertices $S_k \in V$ such that $\rcf{F}{S_k} \neq 0$.

\begin{proposition}
\label{proposition:stdCoeffStrGraph}
Suppose that $F \in \fillcontentc{\lambda}{z}$, with $F=\sum_ {S_i \in \stdcontent{\lambda}{z}}  a_i S_i$ in $R^{\fillcontent{\lambda}{z}} / A(\lambda, z)$. Then 
\begin{center}
$a_i = \displaystyle \sum_{\substack{S_j \in V_F \\ \langle S_{p_d},\ldots,S_{p_0} \rangle \in \mathcal{P}(S_j,S_i)}} (-1)^{d} \cdot \rcf{F}{S_{p_d}} \cdot \rcf{S_{p_d}}{S_{p_{d-1}}} \cdots \rcf{S_{p_{1}}}{S_{p_0}}$.
\end{center}
\end{proposition}
\begin{proof}
Let $b_0<b_1<\cdots<b_d$ be a chain of inequalities indexing a summand from \eqref{eq:summationFormulaStr}. The corresponding summand in \eqref{eq:summationFormulaStr} is nonzero if and only if $\rcf{F}{S_{b_d}},\rcf{S_{b_d}}{S_{b_{d-1}}},\ldots,\rcf{S_{b_1}}{S_{b_0}}$ are all nonzero. A directed edge from a vertex $S_k$ to a vertex $S_l$ in $G$ can only exist if $l \leq k$. Hence $\rcf{F}{S_{b_d}},\rcf{S_{b_d}}{S_{b_{d-1}}},\ldots,\rcf{S_{b_1}}{S_{b_0}}$ are all nonzero if and only if $S_{b_d} \in V_F$ and there is a path $\langle S_{b_d},S_{b_{d-1}},\ldots,S_{b_1},S_{b_0} \rangle$ in $G$.
\end{proof}

\begin{example}
As in the previous example let $\lambda=(3,3,2)$ and $z=(1,2,1,2,2)$. Additionally, let $S_1,\ldots,S_6$ be defined as they were in that example. Let $F \in \fillcontentc{\lambda}{z}$ with
\begin{center}
$\ytableausetup{mathmode,boxsize=1em,centertableaux}
F=\begin{ytableau}
2 & 2 & 1 \\
4 & 3 & 5 \\
5 & 4 \\
\end{ytableau}$
\end{center}

We may find the coefficient $a_1$ of the term $S_1$ in the straightening of $F$ using the graph $G$ from Example \ref{example:graph}. We have $\rcf{F}{S_6}=0, \rcf{F}{S_5}=1,\rcf{F}{S_4}=0,\rcf{F}{S_3}=-1,\rcf{F}{S_2}=-1$, and $\rcf{F}{S_1}=2$. Thus $V_F = \{ S_5, S_3, S_2, S_1 \}$. We have the following paths from these vertices to $S_1$ in the graph $G$.
\begin{center}
$\mathcal{P}(S_5,S_1) = \{ \langle S_5, S_1 \rangle \},\; \mathcal{P}(S_3,S_1) = \emptyset,\; \mathcal{P}(S_2,S_1) = \emptyset,\; \mathcal{P}(S_1,S_1) = \{ \langle S_1 \rangle \}$
\end{center}

Then according to Proposition \ref{proposition:stdCoeffStrGraph} we have two summands, one for each path listed above, giving us
\begin{center}
$\begin{array}{rl}
a_1 & = (-1)^1 \cdot \rcf{F}{S_5} \cdot \rcf{S_5}{S_1} \cdot \rcf{S_1}{S_1}  + (-1)^0 \cdot \rcf{F}{S_1} \cdot \rcf{S_1}{S_1} \\
\;  & = -1 \cdot 1 \cdot 1 \cdot 1 + 1 \cdot 2 \cdot 1 = 1 \\
\end{array}$
\end{center}
\end{example}

\bibliographystyle{alpha}
\bibliography{straightencoeff}

\end{document}